\documentclass[11pt]{article}
\usepackage{graphicx}
\usepackage{amsthm, amsmath, amssymb, tikz, bbm}
\usepackage{mathtools}
\usepackage{xifthen}
\usepackage[colorlinks=true,linkcolor=blue,citecolor=blue, urlcolor=blue]{hyperref}

\usepackage[margin=1in]{geometry} 

\usepackage[shortlabels]{enumitem}
\usepackage{todonotes}
\usetikzlibrary{calc,shapes, backgrounds}

\newtheorem{lemma}{Lemma}
\newtheorem{theorem}{Theorem}
\newtheorem{corollary}{Corollary}

\newtheorem{claim}{Claim}
\newtheorem{remark}{Remark}

\newtheorem{proposition}{Proposition}

\newcommand{\dss}{\displaystyle\sum}

\newcommand{\lp}{\left (}
\newcommand{\rp}{\right )}

\newcommand{\cF}{\mathcal{F}}

\newcommand{\cG}{\mathcal{G}}

\newcommand{\CR}{{\rm cr}}

\newcommand{\commentout}[1]{}

\title{Anti-Ramsey number of edge-disjoint rainbow spanning trees in all graphs}
\author{
Linyuan Lu
\thanks{University of South Carolina, Columbia, SC 29208,
({\tt lu@math.sc.edu}).}
\and
Andrew Meier \thanks{University of South Carolina, Columbia, SC 29208,
({\tt am66@mailbox.sc.edu}).} 
\and
Zhiyu Wang \thanks{Georgia Institute of Technology, Atlanta, GA 30332,
({\tt zwang672@gatech.edu}).}
}

\begin{document}

\maketitle
\begin{abstract}
An edge-colored graph $G$ is called \textit{rainbow} if every edge of $G$ receives a different color. Given any host graph $G$,
the \textit{anti-Ramsey} number of $t$ edge-disjoint rainbow spanning trees in $G$, denoted by $r(G,t)$, is defined as the maximum number of colors in an edge-coloring of $G$ containing no $t$ edge-disjoint rainbow spanning trees.
For any vertex partition $P$, let $E(P,G)$ be the set of non-crossing edges in $G$ with respect to $P$.
In this paper, we determine $r(G,t)$ for all host graphs $G$: $r(G,t)=|E(G)|$ if there exists a partition $P_0$ with $|E(G)|-|E(P_0,G)|<t(|P_0|-1)$; and 
$r(G,t)=\max_{P\colon |P|\geq 3} \{|E(P,G)|+t(|P|-2)\}$ otherwise. As a corollary, we determine $r(K_{p,q},t)$ for all values of $p,q, t$, improving a result of Jia, Lu and Zhang.

%When $G$ is a complete graph $K_n$, $r(K_n,t)$ is completely determined by Lu-Wang \cite{}, and settled a conjecture of
%Jahanbekam and West [{\em J. Graph Theory, 2016}]. Recently Jia, Lu, and Zhang considered the host graph to be the complete bipartite graphs $K_{p,q}$ and proved $r(K_{p,q},1)=(p-2)q+1+\delta_{pq}$ for $p\geq q \geq 4$ and $r(K_{p,q},t)=(p-2)q+ t+\delta_{pq}$ for $p\geq q \geq 2t+\sqrt{3t-2}+4$. Here $\delta_{pq}$ is the Kronecker delta symbol (i.e, $\delta_{pq}=1$ if $p=q$ and 0 otherwise.) In this paper, we completely determined $r(K_{p,q},t)$ for all ranges of $p$, $q$, and $t$.

\end{abstract}

\section{Introduction}
Given a (multi-)graph $G$, let $V(G)$ and $E(G)$ denote the vertex set and edge set of $G$ respectively. \commentout{Let $\|G\|$ denote the number of edges in $G$.}An edge-colored graph $G$ is called \textit{rainbow} if every edge of $G$ receives a different color.
The general \textit{anti-Ramsey problem} asks for the maximum number of colors $AR(K_n,\cG)$ in an edge-coloring of $K_n$ containing no rainbow copy of any graph in a class $\cG$. For some earlier results when $\cG$ consists of a single graph, see the survey \cite{FMO10}. In particular, Montellano-Baallesteros and Neumann-Lara \cite{Montella-Nemann05} showed a conjecture of Erd\H{o}s, Simonovits and S\'os \cite{ESS75} by computing $AR(K_n, C_k)$. Jiang and West \cite{Jiang-West04} determined the anti-Ramsey number of the family of trees with $m$ edges. For some more recent results, see for example \cite{FGLX2021, Gorgol2016, GLS2020, Jiang-Pikhurko2009, LSS2019, Xie-Yuan2020, Yuan2021+, Yuan-Zhang2021+}.

Anti-Ramsey problems have also been investigated for rainbow spanning subgraphs. In particular, Hass and Young \cite{Haas-Young12} showed that the anti-Ramsey number for perfect matchings (when $n$ is even) is $\binom{n-3}{2}+2$ for $n\geq 14$. For rainbow spanning trees, let $r(K_n,t)$ be the maximum number of colors in an edge-coloring of $K_n$ not having $t$ edge-disjoint rainbow spanning trees. Bialostocki and Voxman \cite{Bialostocki-Voxman01} showed that $r(K_n,1) = \binom{n-2}{2}+1$ for $n \geq 4$. Akbari and Alipour \cite{Akbari-Alipour07} showed that $r(K_n,2) = \binom{n-2}{2}+2$ for $n\geq 6$. Jahanbekam and West \cite{Jahanbekam-West16} extended the investigations to arbitrary number of edge-disjoint rainbow spanning trees (along with the anti-Ramsey number of some other edge-disjoint rainbow spanning structures such as matchings and cycles). In particular, for $t$ edge-disjoint rainbow spanning trees, they showed that
	 $$r(K_n,t) = \begin{cases} 
	                     \binom{n-2}{2} +t   & \textrm{ for } n >  2t+\sqrt{6t-\frac{23}{4}}+\frac{5}{2},\\
     			\binom{n}{2}-t & \textrm{ for } n = 2t,
 		      \end{cases}$$
and they \cite{Jahanbekam-West16} conjectured that $r(K_n,t) = \binom{n-2}{2} + t$ whenever $n\geq 2t+2 \geq 6$. This conjecture was recently settled by the first and the third author in \cite{Lu-Wang-AntiRamsey}.
Together with previous results \cite{Bialostocki-Voxman01, Akbari-Alipour07, Jahanbekam-West16}, they give the anti-Ramsey number of $t$ edge-disjoint rainbow spanning trees in $K_n$ for all values of $n$ and $t$. 
\begin{theorem}[Lu-Wang \cite{Lu-Wang-AntiRamsey}]\label{anti-Ramsey}
	For all positive integers $n$ and $t$,
 $$r(K_n,t) = \begin{cases} 
 {n-2 \choose 2}+t   & \textrm{ for } n \geq 2t+2,\\
{n-1\choose 2}   & \textrm{ for } n = 2t+1,\\
{n\choose 2}-t & \textrm{ for } n = 2t,\\
{n\choose 2} &   \textrm{ for } n<2t.
 \end{cases}$$
\end{theorem}

Very recently, Jia, Lu, and Zhang \cite{Jia-Lu-Zhang2021}
considered the analogous problem when the host graph is the complete bipartite graphs $K_{p,q}$. Using similar approaches in \cite{Jahanbekam-West16}, they proved that $r(K_{p,q},1)=(p-2)q+1+\delta_{pq}$ for $p\geq q \geq 4$ and $r(K_{p,q},t)=(p-2)q+ t+\delta_{pq}$ for $p\geq q \geq 2t+\sqrt{3t-2}+4$.  Here $\delta_{pq}$ is the Kronecker delta function:
$$
\delta_{pq}=\begin{cases}
1 & \mbox{ if } p=q,\\
0 & \mbox{ otherwise}.
\end{cases}
$$

Both of the results in \cite{Jahanbekam-West16} and \cite{Jia-Lu-Zhang2021} leave a gap of $\Theta(\sqrt{t})$ when the number of edge-disjoint rainbow spanning trees is closer to  the maximum possible number of edge-disjoint spanning trees in $K_n$ and $K_{p,q}$ respectively. 

In the paper, we show a general framework that enables us to determine the anti-Ramsey number $r(G,t)$  for any host graph $G$ and all values of $t$. 
Given a graph $H$ and a partition of the vertex set of $H$, let $\CR(P,H)$ be the set of crossing edges in $H$ whose end vertices belong to different parts in the partition $P$. Let $E(P,H)$ be the set of non-crossing edges of $H$ with respect to $P$, i.e., edges whose two end vertices are contained in the same part in the partition $P$.
\begin{theorem}\label{main}
For any multi-graph $G$, if there is a partition $P_0$ of vertices of $G$ satisfying $|E(G)|-|E(P_0,G)|<t(|P_0|-1)$, then $r(G,t)=|E(G)|$. Otherwise, 
\begin{equation}\label{eq:main}
r(G,t)=\max_{P\colon |P| \geq 3} \{|E(P,G)|+t(|P|-2)\},
\end{equation}
where the maximum is taken among all partitions $P$
(with $|P|\geq 3$) of the vertex set of $G$.
\end{theorem}

As a corollary, we derive an explicit formula for the anti-Ramsey number of $t$ edge-disjoint rainbow spanning trees in multipartite graphs. For positive integers $n_1, n_2, \ldots, n_r$ with $\sum_{i=1}^r n_i = n$, let 
$K_{n_1, n_2,\ldots, n_r}$ be the complete $r$-partite graphs with vertex parts of size $n_1, n_2,\ldots, n_r$ respectively.
\begin{theorem}\label{Kpq:main}
For any $n_1\geq n_2\geq \cdots \geq n_r>0$ and $n=n_1+n_2+\cdots+n_r$,
we have
\begin{equation}
r(K_{n_1, n_2,\ldots, n_r},t)=
\begin{cases}
|E(K_{n_1, n_2,\ldots, n_r})| &  \hspace*{-30mm}\mbox{ if } t(n-1)> |E(K_{n_1, n_2,\ldots, n_r})| \mbox{ or } \sum_{i=2}^r n_i<t,\\
\max\{t(n-2), |E(K_{n_1-2, n_2,\ldots, n_r})|+ t +\delta_{n_1n_2}\} & \mbox{ otherwise.}
\end{cases}
\end{equation}
\end{theorem}

In particular, Theorem \ref{Kpq:main} settles $r(K_{pq},t)$ completely. We also re-derive Theorem \ref{anti-Ramsey} as a corollary.
\begin{corollary}\label{cor:Kpq}
For any $p\geq q$, we have
\begin{equation}
r(K_{p,q},t)=
\begin{cases}
pq & \mbox{ if } t(p+q-1)>pq  \mbox{ or } q<t,\\
\max\{t(p+q-2), (p-2)q+t+\delta_{pq}\} & \mbox{ otherwise.}
%t(p+q-2) &\mbox{ if } t(p+q-3)>(p-2)q +\delta_{pq},\\
%(p-2)q +t +\delta_{pq} & \mbox{ otherwise}.
\end{cases}
\end{equation}
\end{corollary}

% {\bf Remark:} Note that $f(G,t,P)$ depends
% in the anti-ramsey number of proper induced subgraphs. 
% It is a recursive formula.

% Item 1 is trivial. Since $G$ has no $t$ edge-disjoint spanning trees, $r(G,t)=|E(G,t)|$
% by coloring all edges in distinct colors. 

% We have the following explicit formula for $t=1,2$ and general graph $G$.

% \begin{theorem}\label{t2}
% For any multi-graph $G$ and $t=1$ (or $t=2$), the following formula holds.
% \begin{enumerate}
%     \item If there is a partition $P_0$ of vertices of $G$ satisfying
%     $|\CR(P_0,G)|<t(|P_0|-1)$, then we have $$r(G,t)=|E(G)|;$$
%     \item else, we have
%     \begin{equation}\label{eq:main}
% r(G,t)=\max_P \{ |E(G)|-|\CR(P,G)|+t(|P|-2)\},
% \end{equation}
% where the maximum is taken among all partition $P$ (with $|P|\geq 3$) of the vertex set of $G$.
% \end{enumerate}
% \end{theorem}

The paper is organized as follows. In Section \ref{Notation}, we define the basic machinery needed to prove Theorem \ref{main}. We turn to the proof of Theorem \ref{main} in Section \ref{Proof-of-Main}, and in Section \ref{Derivation} we show Theorems \ref{anti-Ramsey} and \ref{Kpq:main} are direct consequences of Theorem \ref{main}.

\section{Notation and tools}\label{Notation}

Let $G$ be an edge-colored multigraph and $\{ F_1, \ldots, F_t$\} be a collection of $t$ edge-disjoint
rainbow spanning forests in $G$. We say $\{ F_1, \ldots, F_t\}$ are \textit{color-disjoint} if no color appears in the edges of more than one member of $\{ F_1, \ldots, F_t$\}.
We are also interested in whether $F_1,\ldots, F_t$ can be \textit{extended} to
$t$ edge-disjoint rainbow spanning trees $T_1,\ldots, T_t$ in $G$, i.e., $E(F_i)\subseteq E(T_i)$ in $G$ for each $i$. We call such an extension {\em color-disjoint} if all edges in $\cup_i\lp E(T_i)\setminus E(F_i)\rp$
have distinct colors and these colors are different from the colors appearing in the edges of $\cup_i E(F_i)$. 

The main tools we use to show Theorem \ref{main} are two structure theorems that characterize the existence of $t$ color-disjoint rainbow spanning trees or the existence of a color-disjoint extension of $t$ edge-disjoint rainbow spanning forests into $t$ edge-disjoint rainbow spanning trees. When $t=1$, Broersma and Li \cite{Broersma-Li97} showed that determining the largest rainbow spanning forest of a graph can be solved by applying the \textit{Matroid Intersection Theorem}. The following characterization was established by Schrijver \cite{Schrijver03} using matroid methods, and later given graph theoretical proofs by Suzuki \cite{Suzuki06} and also by Carraher and Hartke \cite{Carraher-Hartke17}.

\begin{theorem}{(\cite{Schrijver03, Suzuki06, Carraher-Hartke17})}\label{Suzuki}
An edge-colored connected graph $G$ has a rainbow spanning tree if and only if for every $2\leq k\leq n$ and every partition of $G$ with $k$ parts, at least $k-1$ different colors are represented in edges between partition classes.
\end{theorem}

The above result can be generalized to $t$ color-disjoint rainbow spanning trees in Theorem \ref{partition} using similar matroid methods by Schrijver \cite{Schrijver03} or using graph-theoretical approachs in \cite{Lu-Wang-AntiRamsey}.

\begin{theorem}\cite{Schrijver03, Lu-Wang-AntiRamsey}\label{partition}
An edge-colored multigraph $G$ has $t$ pairwise color-disjoint rainbow spanning trees if and only if for every partition $P$ of $V(G)$ into $|P|$ parts, at least $t(|P|-1)$ distinct colors are represented in edges between partition classes.
\end{theorem}

% \begin{remark}
% Recall the famous Nash-Williams-Tutte Theorem (\cite{Nash-Williams61, Tutte61}): A multigraph contains $t$ edge-disjoint spanning trees if and only if for every partition $P$ of its vertex set, it has at least $t(|P|-1)$ cross-edges. Theorem \ref{partition} implies the Nash-Williams-Tutte Theorem by assigning every edge of the multigraph a distinct color.
% \end{remark}

Given a set of edges $E$, we use $c(E)$ to denote the set of colors that appear in some edge of $E$. We need the following result on the existence of a color-disjoint extension of rainbow spanning forests to rainbow spanning trees.

\begin{theorem}\cite{Lu-Wang-AntiRamsey}\label{extension}
A family of $t$ edge-disjoint rainbow spanning forests $F_1, \ldots, F_t$ 
has a color-disjoint extension in $G$ if and only if for every partition $P$ of $G$ into $|P|$ parts,
\begin{equation}
  \label{eq:ext}
  |c(\CR(P,G'))|+\sum_{i=1}^t|\CR(P, F_i)|\geq t(|P|-1),
\end{equation}
where $G'$ is the spanning subgraph of $G$ by removing all edges with colors appearing in $\{F_1, \ldots, F_t\}$.
\end{theorem}

For the remainder of this section, we may assume that the graph $G$ satisfies for any partition $P$
\begin{equation}\label{eq:colordisjoint}
    |\CR(P,G)|\geq t(|P|-1).
\end{equation}
If not, then $G$ cannot have $t$ edge-disjoint spanning trees. We can then color all edges
in distinct colors and still avoid $t$ edge-disjoint rainbow spanning trees. Thus
$r(G,t)=|E(G)|$ holds. The following lemma proves the lower bounds.
\begin{lemma}\label{general-lowerbound}
Assume Inequalities \eqref{eq:colordisjoint} holds for a multi-graph $G$.
Then, for any partition $P$ of the vertex set $V(G)$ with $|P|\geq 3$, we have
\begin{equation}\label{eq:lowerbound}
r(G,t)\geq |E(P,G)|+t(|P|-2).
\end{equation}
\end{lemma}

\begin{proof}
When $|P|\geq 3$, we can arbitrarily select $t(|P|-2)-1$ crossing edges (with respect to $P$) and color them in distinct colors. Color the rest of crossing edges in a new color $c_0$ and color the remaining non-crossing edges in (a new set of) distinct colors. We claim that $G$ contains no $t$ edge-disjoint rainbow spanning trees. Otherwise, each rainbow spanning tree contains $(|P|-2)$ crossing edges not in color $c_0$. So $G$ has $t(|P|-2)$ crossing edges not in color $c_0$, giving a contradiction.
\end{proof}

Now we assume $G$ is equipped with an edge coloring with $\max_{P\colon|P|\geq3}\{|E(P,G)|+t(|P|-2)\}+1$ colors.
For a partition $P$ of the vertices of $G$, we define
\begin{align*}
\eta(P,G)&:=|E(P,G)|-|c(E(P,G))|\\
\shortintertext{and}
\xi(P,G)&:=|c(E(P,G))\cap c(\CR(P,G))|.
\end{align*}

\noindent
Applying the inclusion-exclusion principle on the two sets $c(E(P,G))$ and $c(\CR(P,G))$, we have
\begin{align*}
|c(E(G))| -|E(P,G)| &=
|c(E(P,G)) \cup c(\CR(P,G))| -|E(P,G)|\\
&= |c(E(P,G))|+|c(\CR(P,G))|-|c(E(P,G)) \cap c(\CR(P,G))| -|E(P,G)|\\
&=|c(\CR(P,G))| -\eta(P,G) -\xi(P,G).
\end{align*}

It follows that 
\begin{equation}\label{eq:CRPG}
 |c(\CR(P,G))| = |c(E(G))| -|E(P,G)| +\eta(P,G) +\xi(P,G). 
\end{equation}
Since $c(E(G)) \geq \max_{P\colon|P|\geq 3}\{|E(P,G)|+t(|P|-2)\}+1$,   we then have the following corollary.
\begin{corollary}\label{crossingcolors} For any non-trivial partition $P$ of $V(G)$, we have
\begin{enumerate}
    \item if $|P|\geq 3$, then
    $$|c(\CR(P,G))|\geq t(|P|-2)+\eta(P,G) + \xi(P,G) +1;$$
    \item if $|P|=2$ and $|c(E(G)|>|E(P,G)|$, then
   $$|c(\CR(P,G))|\geq \eta(P,G) + \xi(P,G) +1.$$  
\end{enumerate}
\end{corollary}

% Given an edge-colored graph $G$, we call a partition $P$ of $V(G)$ \emph{deficient} if $|c(\CR(P,G))|\leq t(|P|-1)-1$. \commentout{Moreover, given a family of edge-disjoint rainbow spanning forest $\cF = \{F_1, \ldots, F_t\}$ and a partition $P$ of $V(G)$, define the \textit{deficiency} of $P$ w.r.t. $\cF$, denoted by $D_{\cF}(P)$, as 

% $$D_{\cF}(P) = \max\{0, t(|P|-1) -|c(\CR(P,G'))| - \sum_{i=1}^t |\CR(P,F_i)|)\}.$$}
% By Theorem \ref{partition}, if no deficient bipartitions exist, then $G$ contains $t$ edge-disjoint rainbow spanning trees. 

We also need the following lemma on when a family of $t$ edge-disjoint rainbow spanning forests satisfies the Inequality \eqref{eq:ext} in Theorem \ref{extension}.

\begin{lemma}\label{Fi-large-enough2} Let $\cF = \{F_1,F_2,\ldots,F_t\}$ be a set of $t$ edge-disjoint rainbow spanning forests such that $\sum\limits_{i=1}^t|F_i|\geq t-1+|c(\cup_{i=1}^t F_i)|$, and each color appearing in the $F_i$ appears at least twice. Then for any partition satisfying either hypothesis in Corollary \ref{crossingcolors}, $$|c(\CR(P,G'))|+\sum\limits_{i=1}^t|\CR(P,F_i)|\geq t(|P|-1).$$ Here, $G'$ is the spanning subgraph of $G$ built by removing all edges with colors appearing in $\{F_1, \ldots, F_t\}$.
\end{lemma}

\begin{proof}

Observe that given any partition $P$, 
$$\dss_{i=1}^t |F_i| = \dss_{i=1}^t |\CR(P, F_i)| + \dss_{i=1}^t |E(P, F_i)|.$$
Since $\sum\limits_{i=1}^t|F_i|\geq t-1+|c(\cup_{i=1}^t F_i)|$, it follows that 
\begin{equation}
    \dss_{i=1}^t |\CR(P, F_i)| \geq (t-1) + |c(\displaystyle\cup_{i=1}^t F_i)| - \dss_{i=1}^t |E(P, F_i)|. \label{eq:cl1a}
\end{equation}

%\begin{claim}\label{cl:edge-equality}
%$$\dss_{i=1}^t |E(P, F_i)| + |c(\CR(P,\cup_{i=1}^t F_i))| = %|c(\displaystyle\cup_{i=1}^t F_i)| + \eta(P,G) + \xi(P,G) . $$
%\end{claim} 
%\begin{proof}
%Let $c_j$ be a color appearing in $\cup_{i=1}^t F_i$. Let $m_j$ %denote the number of edges of $F_1 \cup \ldots \cup F_t$ of color %$c_j$ within the parts of $P$.
%If $c_j$ only appears in $cr(P, \cup_{i=1}^t F_i)$, then it is %counted exactly once in both the L.H.S and the R.H.S. If $c_j$ %appears in both $cr(P, \cup_{i=1}^t F_i)$ and inside the parts of %$P$, then both the R.H.S. and the L.H.S. count the number of %$c_j$-colored edges $(1 + m_j)$ times. If $c_j$ only appears %within the parts of $P$, then both sides count the number of %$c_j$-colored edges $m_j$ times. This completes the proof of the %claim.
 %\end{proof}
Hence it follows that 
\begin{align*}
&\hspace*{-1cm} |c(\CR(P,G'))|+\sum_{i=1}^t |\CR(P,F_i)|- t(|P|-1) \\
 &=  |c(\CR(P,G))|- |c(\CR(P,\cup_{i=1}^t F_i))|+\sum_{i=1}^t|\CR(P,F_i)| - t(|P|-1) \\
 &\geq  \lp t(|P|-2)  + \eta(P,G) + \xi(P,G)+1\rp  - |c(\CR(P,\cup_{i=1}^t F_i))| \hspace*{1cm} \mbox{ by  Corollary \ref{crossingcolors}}  
 \\
   & \qquad+ \lp (t-1) + |c(\displaystyle\cup_{i=1}^t F_i)| -  \dss_{i=1}^t |E(P, F_i)| \rp - t(|P|-1) 
   \hspace*{1cm}
\mbox{ and Inequality \eqref{eq:cl1a}}   \\
 &=  |c(\displaystyle\cup_{i=1}^t F_i)| + \eta(P,G) + \xi(P,G) - \lp |E(P, \cup_{i=1}^t F_i)| + |c(\CR(P,\cup_{i=1}^t F_i))| \rp \\
 &\geq |c(\cup_{i=1}^t F_i)| + \eta(P,\cup_{i=1}^t F_i) + \xi(P,\cup_{i=1}^t F_i) - \lp |E(P, \cup_{i=1}^t F_i)| + |c(\CR(P,\cup_{i=1}^t F_i))| \rp \\
 &=  0.
\end{align*}
The equation in last step is essentially Equation \eqref{eq:CRPG} by replacing $G$ with $\cup_{i=1}^t F_i$.
This completes the proof of the lemma.
\end{proof}
\noindent

\section{Proof of Theorem \ref{main}}\label{Proof-of-Main}

\begin{proof}[Proof of Theorem \ref{main}]
For the sake of contradiction, let $G$ be an edge-minimal counter-example that satisfies the conditions in \eqref{eq:colordisjoint} but does not satisfy
$$r(G,t)\leq \max_{P\colon|P|\geq 3}\{|E(P,G)|+t(|P|-2)\},$$
i.e., there exists a coloring of the edges of $G$ with 
$|c(E(G))|=\max_{P\colon|P|\geq 3}\{|E(P,G)|+t(|P|-2)\}+1$ with no $t$ edge-disjoint rainbow spanning trees. Fix such a coloring of the edges of $G$.
Let $G_2$ be the subgraph of $G$ consisting of all edges whose color appears more than once in $G$. 

\begin{claim}\label{BaseCase}
For every partition $P$ of $G$ with $|P|\geq 3$, we have $|\CR(P,G)|\geq t(|P|-1)+1$.
\end{claim}

\begin{proof}
Suppose for contradiction that there exists a partition $P$ of $G$ with $|P|\geq 3$ and $|\CR(P,G)|=t(|P|-1).$ Then we have
\begin{align}
|c(E(G))| & =
\max_{P\colon|P|\geq 3}\{|E(P,G)|+t(|P|-2)\}+1  \nonumber \\
&\geq |E(G)|-t(|P|-1)+t(|P|-2)+1  \nonumber \\
&=|E(G)|-t+1.  \label{eq:minim_num_colors}
\end{align}
Let $P'$ be any bipartition. By Inequality \eqref{eq:colordisjoint}, we have that $|\CR(P',G)|\geq t$. It follows from \eqref{eq:minim_num_colors} that
\begin{eqnarray*}
|E(P',G)|\leq |E(G)|-t<|c(E(G))|.
\end{eqnarray*}
%It follows that $\max_{|P|\geq 3}\{|E(P,G)|+t(|P|-2)\}\geq |E(P',G)|$ for any bipartition. 
It follows that the second hypothesis of Corollary \ref{crossingcolors} holds for all bipartitions. 

We now show that we can construct $t$ edge-disjoint rainbow spanning forests $F_1, \ldots, F_t$ which satisfy Theorem \ref{extension}. For convenience, let $F_{t+i} \equiv F_i$. We say a color $c$ has \textit{multiplicity} $k$ in $G$ if the number of edges with color $c$ in $G$ is $k$.
Let $m_1\geq m_2\geq\ldots\geq m_s\geq 2$ be the multiplicities of all colors $c_1, c_2,\ldots, c_s$ respectively that have multiplicity at least two in $G$. We have a few cases:

\begin{description}

\item Case 1: $m_1\geq t$. In this case, place one edge of color $c_1$ in each of the forests $F_1, \ldots, F_t$. We have
$$\sum_{i=1}^t |F_i|=t=(t-1)+|c\left(\cup_{i=1}^t F_i\right)|.$$ 
By Lemma \ref{Fi-large-enough2}, $G$ has $t$ edge-disjoint rainbow spanning trees, contradicting the choice of $G$.

\item Case 2:  $m_1\leq t-1$. In this case, we can place one edge of color $c_1$ in each of $\{F_1,F_2,\ldots,F_{m_1}\}$. Then in a similar fashion, place one edge of color $c_2$ in each of $F_{m_1+1}, \ldots, F_{m_1 + m_2}$ (with $F_{t+i}\equiv F_i$). Perform this greedy construction for all colors in $G_2$ (in the order of decreasing multiplicity) until all edges of $G_2$ have been added to $F_1, F_2, \ldots, F_t$ or until each $F_i$ receives two edges. 
If all edges of $G_2$ have been added to $F_1, F_2, \ldots, F_t$, then we have
$$|c(\CR(P,G'))|+\sum_{i=1}^t|\CR(P,F_i)|=|\CR(P,G)|\geq t(|P|-1),$$
where $G'$ is the subgraph of $G$ containing edges with colors of multiplicity $1$.
Hence by Theorem \ref{extension}, $G$ contains $t$ edge-disjoint rainbow spanning trees, contradicting the choice of $G$.
Otherwise, each $F_i$ has exactly two edges. Since there are at most $t$ colors appearing in the $F_i$s, we then have
$$\sum_{i=1}^t|F_i|\geq (t-1)+c\left(\cup_{i=1}^tF_i\right),$$
in which case by Lemma \ref{Fi-large-enough2}, $G$ contains $t$ edge-disjoint rainbow spanning trees, contradicting the choice of $G$ again.
\end{description} 
\end{proof}

\begin{claim}\label{inductivestep}  
There exists an edge $e\in G_2$ whose removal maintains that for every partition $P$, $|\CR(P,G\setminus e)|\geq t(|P|-1).$
\end{claim}

\begin{proof}
By Claim \ref{BaseCase} every partition $P$ of $G$ with $|P|\geq 3$ satisfies $|\CR(P,G)|\geq t(|P|-1)+1$. It follows that $\max_{P\colon|P|\geq 3}\{|E(P,G)|+t(|P|-2)\}\leq |E(G)|-t-1$. Hence the number of colors in $G$ is at most $|E(G)|-t$. Let $C_1$ be the number of colors with multiplicity one in $G$. Since the total number of colors on the edges of $G$ is less than the total number of edges, $C_1\leq |E(G)|-t-1$. We then have that
\begin{eqnarray*}
|E(G_2)|&=&|E(G)|-C_1\\
&\geq&|E(G)|-(|E(G)-t-1)\\
&=&t+1.
\end{eqnarray*}
Since every partition $P$ with $|P|\geq 3$ satisfies that $|\CR(P,G)|\geq t(|P|-1)+1$, we can remove any edge and still satisfy Claim \ref{inductivestep} for all partitions $P$ with $|P|\geq 3$. As such, the only way that Claim \ref{inductivestep} breaks upon removing an edge $e'$ is if $e'\in \CR(P,G)$ with $|P|=2$ and $|\CR(P,G)|=t$. We claim there can only be one such bipartition. 

Suppose we have at least two bipartitions $P_1$ and $P_2$ with $|\CR(P_1,G)|=|\CR(P_2,G)|=t$. Then, $|\CR(P_1\cap P_2,G)| \leq |\CR(P_1,G)| + |\CR(P_2,G)|\leq 2t$. Since $|P_1\cap P_2|\geq 3$, we have
\begin{eqnarray*}
|\CR(P_1\cap P_2, G)|\leq t(|P_1\cap P_2|-1),
\end{eqnarray*}
which contradicts Claim \ref{BaseCase}. As such, only one such bipartition can exist. Now since $|E(G_2)|\geq t+1$, there is an $e\in G_2$ which does not cross this bipartition. It is clear that this edge $e$ satisfies Claim \ref{inductivestep}.
\end{proof}

Now, remove such an edge $e\in G_2$ and consider $G\setminus e$. Since $G$ is an edge-minimal counter-example which satisfies the Inequalities \eqref{eq:colordisjoint} but does not satisfy 
$$r(G,t)\leq \max_{P\colon|P|\geq 3}\{|E(P,G)|+t(|P|-2)\},$$ 
and $G\setminus e$ still satisfies the Inequalities \ref{eq:colordisjoint}, it must happen that
\begin{eqnarray*}
r(G\setminus e,t)&\leq&\max_{P\colon|P|\geq 3}\{|E(P,G\setminus e)|+t(|P|-2)\}\\
&\leq &\max_{P\colon|P|\geq 3}\{|E(P,G)|+t(|P|-2)\}.
\end{eqnarray*}
Because $G$ is colored with $\max_{P\colon|P|\geq 3}\{|E(P,G)|+t(|P|-2)\}+1$ colors and the color of $e$ has multiplicity at least 2, $G\setminus e$ is still colored with $\max_{P\colon|P|\geq 3}\{|E(P,G)|+t(|P|-2)\}+1 \geq r(G\setminus e,t)+1$ colors. It follows that $G\setminus e$, and hence $G$, contains $t$ edge-disjoint rainbow spanning trees, contradicting the choice of $G$. As such, no edge-minimal counter example exists and thus for any graph $G$ in which every partition satisfies Inequalities \eqref{eq:colordisjoint}, 
$$r(G,t)\leq \max_{P\colon|P|\geq 3}\{|E(P,G)|+t(|P|-2)\}.$$
Together with Lemma \ref{general-lowerbound}, this completes the proof of the theorem.
\end{proof}

\section{Derivation of $r(K_n,t)$ and $r(K_{n_1, n_2, \ldots, n_r},t)$}\label{Derivation}
Given a graph $G$ and an integer $s$, we define $$f_G(s):=\max\limits_{P\colon|P|=s}\{|E(P,G)|\}.$$
Theorem \ref{main} can be rephrased as follows.
$$r(G,t) = \begin{cases}
            |E(G)| & \textrm{ if } \displaystyle\max_{2\leq s \leq n} \{f_G(s)+t(s-1)\} >|E(G)|, \\
            \displaystyle\max_{3\leq s\leq n} \{f_G(s)+t (s-2)\}&  \textrm { otherwise.}
           \end{cases}$$
When $G$ is a complete $r$-partite graph and $s \in [a,b]$, we will show that both functions $f_G(s)+t(s-1)$ and  $f_G(s)+t (s-2)$
attain their maximum at the boundary of the interval. We have the following observation. The proof is straightforward and will be omitted here.

\begin{proposition}\label{prop:maximum_partition}
Suppose that $G$ is a complete multi-partite graph. Then, among all partitions of size $s$, the maximum
of $|E(P,G)|$ can be achieved by a partition $P_0$ where $P_0$ can be constructed by repeatedly splitting one vertex from the largest part $(s-1)$ times.
\end{proposition}

In the following lemma, we show that $f_G(s)$ is concave upward in $s$.
\begin{lemma}\label{lem:concave-up}
 Suppose $G$ is a complete $r$-partite graph. Then $f_G(s)$ is concave upward in $s$.
\end{lemma}
\begin{proof}
Let $G = K_{n_1, n_2, \ldots, n_r}$ where $n_1\geq n_2 \geq \ldots \geq n_r$ and $\sum_{i=1}^r n_i = n$.
By Proposition \ref{prop:maximum_partition}, among all partitions of $V(G)$, $f_G(s)$ is maximized when we start with $G$ and then repetitively split one vertex from the largest part $(s-1)$ times. Suppose $K_{n_1', n_2', \ldots, n_r'}$ is the complete $r$-partite graph induced by the largest part in some partition that achieves $f_G(s)$. We can assume $n_1'\geq n_2'\geq \cdots\geq n_r'$. Observe that
\begin{align*}
   f_G(s)&=|E(K_{n_1',n_2',\ldots, n_r'})|,\\
   f_G(s+1)&=|E(K_{n_1'-1,n_2',\ldots, n_r'})|,\\
   f_G(s+2)&=\begin{cases}
   |E(K_{n_1'-2,n_2',\ldots, n_r'})| & \mbox{ if } n_1'>n_2',\\
   |E(K_{n_1'-1,n_2'-1,\ldots, n_r'})| & \mbox{ if } n_1'=n_2'.
   \end{cases}\\
    &= |E(K_{n_1'-2,n_2',\ldots, n_r'})| +\delta_{n_1'n_2'}.
\end{align*}
  Hence we have that
  \begin{align*}
      f_G(s)-2f_G(s+1)+f_G(s+2)
      = \delta_{n_1'n_2'}\geq 0.
  \end{align*}
Thus,  $f_G(s)+f_G(s+2)\geq 2 f_G(s+1)$ for all $2 \leq s \leq n-2$. It follows that 
$f_G(s)$ is concave upward in $s$.
\end{proof}

\begin{proof}[Proof of Theorem 3]
Let $G = K_{n_1, n_2, \ldots, n_r}$ where $n_1\geq n_2 \geq \ldots \geq n_r$ and $\sum_{i=1}^r n_i = n$. Similar to Lemma \ref{lem:concave-up}, we have that
\begin{align*}
f_G(2)&= |E(K_{n_1-1,n_2,\ldots, n_r})|,\\
    f_G(3)&=|E(K_{n_1-2,n_2,\ldots, n_r})| +\delta_{n_1n_2},\\
    f_G(n)&=0.
\end{align*}
By Lemma \ref{lem:concave-up}, $f_G(s)$ is concave upward in $s$. By the properties of convex functions, the functions $f_G(s)+ t(s-1)$ and $f_G(s)+ t(s-2)$ are also concave upward in $s$. It follows that the maximum values of these functions are achieved at the ends of the interval (when $s$ ranges over an interval domain). Hence
\begin{align*}
 \max_{2\leq s \leq n} \{f_G(s)+t(s-1)\}
 &=\max\{f_G(2)+t, f_G(n)+t(n-1)\}\\
 &=\max \{ |E(K_{n_1-1,n_2,\ldots, n_r})|+t, t(n-1)\}.
\end{align*}
Thus, $\max_{2\leq s \leq n} \{f_G(s)+t(s-1)\}>|E(G)|$ if and only if
$t(n-1)> |E(G)|$ or $\sum_{i=2}^r n_i<t$.
Similarly, 
\begin{align*}
 \max_{3\leq s \leq n} \{f_G(s)+t(s-2)\}
 &=\max\{f_G(3)+t, f_G(n)+t(n-2)\}\\
 &=\max \{ |E(K_{n_1-2,n_2,\ldots, n_r})|+t +\delta_{n_1n_2}, t(n-2)\}.
\end{align*}
This completes the proof of Theorem \ref{Kpq:main}.
\end{proof}

\begin{remark}
Now we re-derive Theorem \ref{anti-Ramsey} as a corollary of Theorem \ref{main}.
Suppose $G=K_n$. Similar to Lemma \ref{lem:concave-up}, it is straightforward to check that $f_G(s)$ is also concave upward in $s$ when $G=K_n$. So are the functions $f_G(s)+ t(s-1)$ and $f_G(s)+ t(s-2)$. Observe that
\begin{align*}
    f_G(2) &= \binom{n-1}{2},\\
    f_G(3) &= \binom{n-2}{2},\\
    f_G(n) &= 0.
\end{align*}
Hence \begin{align*}
 \max_{2\leq s \leq n} \{f_G(s)+t(s-1)\}
 &=\max\{f_G(2)+t, f_G(n)+t(n-1)\}\\
 &=\max \{\binom{n-1}{2}+t, t(n-1)\}.
\end{align*}

Now the condition $t(n-1)>{n\choose 2}$ can be simplified to $n<2t$, and the condition $\binom{n-1}{2}+t > \binom{n}{2}$ can be simplified to $n< t+1$. Thus
$r(K_n,t)={n\choose 2}$ if $n<2t$.
When $n\geq 2t$, we have
\begin{align*}
r(K_n,t) &=
   \max\{t(n-2), \binom{n-2}{2} +t \}\\
&= \begin{cases}
t(n-2) & \mbox{ if } n=2t,\\
{n-2\choose 2} +t  & \mbox{ if } n\geq 2t+1.
\end{cases}\\
&=  \begin{cases}
{n\choose 2}-t & \mbox{ if } n=2t,\\
{n-1\choose 2} & \mbox{ if } n=2t+1,\\
{n-2\choose 2} +t &\mbox{ if } n\geq 2t+2.
\end{cases}
\end{align*}
\end{remark}

\begin{remark}
Now consider the complete bipartite graph $K_{pq}$.
By Theorem \ref{Kpq:main}, we have that for any $p\geq q$,
$$
r(K_{p,q},t)=
\begin{cases}
pq & \mbox{ if } t(p+q-1)>pq  \mbox{ or } t>q,\\
\max\{t(p+q-2), (p-2)q+t+\delta_{pq}\} & \mbox{ otherwise.}
\end{cases}
$$
In particular, when $p\geq q\geq 2t+1$,
we have $$pq-t(p+q-1)=(p-t)(q-t)-t^2+t> 0,$$
and 
\begin{align*}
    (p-2)q+t+\delta_{pq}-t(p+q-2)
    &= (p-2-t)(q-t) -t^2+t + \delta_{pq}\\
    &\geq (q-t-1)^2 -t^2+t\\
    &\geq t.
\end{align*}
Thus, for $p\geq q\geq 2t+1$,
$$r(K_{p,q},t)= (p-2)q+t+\delta_{pq}.$$
This implies and improves Jia, Lu and Zhang's result in \cite{Jia-Lu-Zhang2021}.
\end{remark}


\begin{thebibliography}{1}

\bibitem{Akbari-Alipour07}
S. Akbari, A. Alipour, Multicolored trees in complete graphs, \textit{J. Graph Theory}, {\bf 54}, (2007), 221-232.

\bibitem{Bialostocki-Voxman01}
 A. Bialostocki and W. Voxman, On the anti-Ramsey numbers for spanning trees. \textit{Bull. Inst. Combin. Appl.} {\bf 32} (2001), 23-26.

\bibitem{Broersma-Li97}
H. Broersma, X. Li, Spanning trees with many or few colors in edge-colored graphs, \textit{Discuss. Math. Graph Theory} \textbf{17(2)} (1997), 259-269.

\bibitem{Carraher-Hartke17}
J. M. Carraher and S. G. Hartke, Eulerian circuits with no monochromatic transitions in edge-colored digraphs
with all vertices of outdegree three, \textit{SIAM J. Discrete Math.} {\bf 31}(1) (2017), 190-209.

\bibitem{CHH16}
J. Carraher, S. Hartke, P. Horn, Edge-disjoint rainbow spanning trees in complete graphs, \textit{European J. Combin.}, {\bf 57} (2016), 71-84.

\bibitem{Edmonds68}
J. Edmonds, Matroid Partition, \textit{Math. of the Decision Sciences}, Amer. Math Soc. Lectures in Appl. Math. \textbf{11} (1968), 335-345.


\bibitem{Edmonds70}
J. Edmonds, Submodular functions, matroids and
certain polyhedra. in \textit{Combinatorial structures and
their applications}, eds. R. Guy, H. Hanani, N.
Sauer and J. Schonheim, Pages 69-87, 1970.


\bibitem{ESS75}
P. Erd\H{o}s, M. Simonovits, and V. S\'os, Anti-Ramsey theorems, in \textit{Infinite and Finite Sets} (Colloq. Keszthely 1973). Colloq Math Soc Janos Bolyai \textbf{10} (1975), 633-643.

\bibitem{FGLX2021}
 C. Fang, E. Gy\H{o}ri, M. Lu and J. Xiao, On the anti-Ramsey number of forests, \textit{Discrete Appl. Math.} \textbf{291} (2021), 129-142.

\bibitem{FMO10}
S. Fujita, C. Magnant, and K. Ozeki, Rainbow generalizations of Ramsey
theory: a survey. \textit{Graphs Combin} \textbf{26} (2010), 1-30.

\bibitem{Gorgol2016}
I. Gorgol, Anti-Ramsey numbers in complete split graphs. \textit{Discrete Math.} \textbf{339} (2016), 1944-1949.

\bibitem{GLS2020}
 R. Gu, J. Li and Y. Shi, Anti-Ramsey numbers of paths and cycles in hypergraphs,
\textit{SIAM J. Discrete Math.} \textbf{34(1)} (2020), 271-307.

\bibitem{Haas-Young12}
R. Haas and M. Young, The anti-Ramsey number of perfect matching, \textit{Discrete Math.}, {\bf 312} (2012), 933-937.

\bibitem{Jahanbekam-West16}
S. Jahanbekam, D.B. West, Anti-Ramsey problems for t edge-disjoint rainbow spanning subgraphs: cycles, matchings, or trees,  \textit{J Graph Theory} \textbf{82(1)} (2016), 75-89.

\bibitem{Jia-Lu-Zhang2021}
Y. Jia, M. Lu, Y. Zhang, Anti-Ramsey Problems in Complete Bipartite Graphs for t Edge-Disjoint Rainbow Spanning Trees, \textit{Graphs Combin.} \textbf{37} (2021), 409–433.

\bibitem{Jiang-Pikhurko2009}
T. Jiang and O. Pikhurko, Anti-Ramsey numbers of doubly edge-critical graphs, \textit{J. Graph Theory.} \textbf{61} (2009) 210-218.


\bibitem{Jiang-West04}
T. Jiang and D. B. West, Edge-colorings of complete graphs that avoid polychromatic
trees, \textit{Discrete Math.} \textbf{274} (2004), 137-145.

\bibitem{LSS2019}
Y. Lan, Y. Shi and Z. Song, Planar anti-Ramsey numbers of paths and cycles, \textit{Discrete Math.}, \textbf{342} (2019), 3216-3224.

\bibitem{Lu-Wang-AntiRamsey}
L. Lu and Z. Wang, Anti-Ramsey number of edge-disjoint rainbow spanning trees,
\textit{SIAM J. Discrete Math.} {\bf34(1)} (2020), 271-307.

\bibitem{Montella-Nemann05}
J. J. Montellano-Ballesteros and V. Neumann-Lara, An anti-Ramsey theorem
on cycles. \textit{Graphs Combin} \textbf{21} (2005), 343-354.

\bibitem{Nash-Williams67}
C. St. J. A. Nash-Williams, An application of matroids to graph theory, in: \textit{Theory of Graphs - International Symposium} (Rome, 1966; P. Rosenstiehl, ed.), Gordon and Breach, New York, and Dunod, Paris, 1967, pp. 263-265.

\bibitem{Nash-Williams61}
C. St. J. A. Nash-Williams, Edge disjoint spanning trees of finite graphs. \textit{J. London Math. Soc.}, {\bf 36} (1961), 445-450.

\bibitem{Schrijver03}
A. Schrijver, \textit{Combinatorial optimization. Polyhedra and efficiency. Vol. B}, Volume 24 of Algorithms and Combinatorics. Springer-Verlag, Berlin, 2003. Matroids, trees, stable sets, Chapters 39-69.

\bibitem{Suzuki06}
K. Suzuki, A necessary and sufficient condition for the existence of a heterochromatic spanning tree in a graph, \textit{Graphs Combin.} {\bf 22(2)} (2006), 261-269.

\bibitem{Tutte61}
W. T. Tutte, On the problem of decomposing a graph into $n$ connected factors, \textit{Journal London Math. Soc}, {\bf 142} (1961), 221-230.

\bibitem{Xie-Yuan2020}
 T. Xie and L. Yuan, On the anti-Ramsey numbers of linear forests, \textit{Discrete Math.}, \textbf{343} (2020), 112130.

\bibitem{Yuan2021+}
L. Yuan, The anti-Ramsey number for paths, \textit{arXiv:2102.00807}.

\bibitem{Yuan-Zhang2021+}
 L. Yuan and X. Zhang, Anti-Ramsey numbers of graphs with some decomposition
family sequences, \textit{arXiv:1903.10319}.


\end{thebibliography}
\end{document}